\documentclass[a4paper]{amsart}

\pdfoutput=1

\usepackage{amsmath,amssymb,amsthm}
\usepackage[utf8]{inputenc}
\usepackage{mathtools}
\usepackage{scalerel}
\usepackage{multicol}
\usepackage{enumerate}
\usepackage{comment} 
\usepackage{microtype}
\usepackage{appendix}
\usepackage[dvipsnames,svgnames,x11names,hyperref]{xcolor}
\definecolor{darkred}{RGB}{139,0,0}
\definecolor{darkblue}{RGB}{0,0,139}
\definecolor{darkgreen}{RGB}{0,100,0}
\usepackage[linktocpage]{hyperref}
\hypersetup{
  colorlinks   = true, 
  urlcolor     = darkred, 
  linkcolor    = darkred, 
  citecolor   = darkgreen}
\usepackage{tikz}
\usepackage{tikz-cd}
\usepackage{bbm}
\usepackage{verbatim}
\usepackage{mathrsfs}

\makeatletter
\let\@wraptoccontribs\wraptoccontribs
\makeatother

\usepackage{spectralsequences}

\newtheorem{thm}{Theorem}[section]
\newtheorem{cor}[thm]{Corollary}
\newtheorem{lem}[thm]{Lemma}

\theoremstyle{definition}
\newtheorem{defn}[thm]{Definition}
\theoremstyle{remark}
\newtheorem{rem}[thm]{Remark}
\newtheorem{example}[thm]{Example}
\numberwithin{equation}{section}

\newcommand{\bK}{\mathbb{K}}
\newcommand{\bL}{\mathbb{L}}

\newcommand{\bN}{\mathbb{N}}

\newcommand{\bQ}{\mathbb{Q}}

\newcommand{\bZ}{\mathbb{Z}}

\newcommand{\gA}{\bold{A}}

\newcommand{\gC}{\bold{C}}

\newcommand{\gE}{\bold{E}}

\newcommand{\gI}{\bold{I}}

\newcommand{\gM}{\bold{M}}

\newcommand{\gR}{\bold{R}}

\newcommand{\cA}{\mathcal{A}}

\newcommand{\cO}{\mathcal{O}}

\newcommand{\fS}{\mathfrak{S}}

\newcommand\lra{\longrightarrow}

\newcommand\GL{\mathrm{GL}}
\newcommand\RBS{\mathrm{RBS}}
\newcommand\rk{\mathrm{rk}}
\newcommand\Tor{\mathrm{Tor}}

\newcommand{\bk}{\mathbbm{k}}

\title[Stability for general linear groups over Dedekind domains]{Homological stability for\\general linear groups over Dedekind domains}

\author{Oscar Randal-Williams}
\email{o.randal-williams@dpmms.cam.ac.uk}
\address{Centre for Mathematical Sciences\\
Wilberforce Road\\
Cambridge CB3 0WB\\
UK}

\begin{document}

\begin{abstract}
We prove a new kind of homological stability theorem for automorphism groups of finitely-generated projective modules over Dedekind domains, which takes into account all possible stabilisation maps between these, rather than only stabilisation by the free module of rank 1. We show the same kind of stability holds for Clausen and Jansen's reductive Borel--Serre spaces.
\end{abstract}

\maketitle

\section{Introduction}

The title of this note requires explanation, for Charney \cite{Charney} has already shown that general linear groups over Dedekind domains enjoy homological stability. Our goal is to show how these results may be obtained, and improved, by using the machinery of cellular $E_k$-algebras which we developed together with Galatius and Kupers \cite{e2cellsI}. More precisely, in Section 18.2 of that paper it was explained how the case of Dedekind domains of class number 1 (i.e.\ PID's) may be treated, and here we analyse the general case. Rather than completeness, our motivation is that a new \emph{kind} of homological stability theorem is appropriate when the class number is $> 1$: we will establish a generic homological stability theorem of this kind. We will also show that it applies to Clausen and Jansen's reductive Borel--Serre spaces.

\subsection{Modules over a Dedekind domain}

Let $\cO$ be a Dedekind domain, and $\mathrm{Pic}(\cO)$ denote the abelian group of isomorphism classes of $\otimes$-invertible $\cO$-modules (which are the rank 1 projective modules; this group may be identified with the ideal class group of $\cO$). The classification of finitely-generated modules over a Dedekind domain has the following consequence: any finitely-generated projective $\cO$-module has the form $M \cong \cO^{n-1} \oplus L$ for some $[L] \in \mathrm{Pic}(\cO)$. In other words finitely-generated projective $\cO$-modules $M$ are classified up to isomorphism by the data of
\begin{enumerate}[(i)]
\item their rank $n = \mathrm{rk}(M)$, and 
\item the isomorphism class $[\Lambda^n_\cO M] \in \mathrm{Pic}(\cO)$ of their top exterior power.
\end{enumerate}

\subsection{Earlier results}

Unless otherwise mentioned, we consider integral homology. For an $\cO$-module $M$, let $\GL(M)$ denote the group of $\cO$-module automorphisms of $M$. Charney's theorem states that for finitely-generated projective $\cO$-modules $M$ and $N$, the homomorphism $- \oplus \mathrm{Id}_N : \GL(M) \to \GL(M \oplus N)$ induces 
\begin{equation}\label{eq:GeneralStab}
(- \oplus \mathrm{Id}_N)_* : H_d(\GL(M) ) \lra H_d(\GL(M \oplus N))
\end{equation}
which is an isomorphism for $d \leq \tfrac{\mathrm{rk}(M)-5}{4}$ and an epimorphism for $d \leq \tfrac{\mathrm{rk}(M)-1}{4}$ (with an improvement if $\cO$ is a PID). Later, van der Kallen \cite[Theorem 4.11]{vdK} has shown that 
$$(- \oplus \mathrm{Id}_\cO)_* : H_d(\GL(\cO^{n}) ) \lra H_d(\GL(\cO^{n+1}))$$
is an isomorphism for $d \leq \tfrac{n-1}{2}$ and an epimorphism for $d \leq \tfrac{n}{2}$. His results are for very general rings, but only treat free modules. These methods were extended by Friedrich \cite[Theorem 2.9]{Friedrich} to treat \emph{arbitrary} (not even projective!) modules over the same class of rings: in the case of Dedekind domains she shows\footnote{In her paper the ``rank of an $\cO$-module $M$'' denotes the largest $m$ such that $\cO^m$ is a direct summand of $M$, so by the classification of finitely-generated projective $\cO$-modules it is usually one less than what we are calling the rank (though they agree for free modules).} that
$$(- \oplus \mathrm{Id}_\cO)_* : H_d(\GL(M) ) \lra H_d(\GL(M \oplus \cO))$$
is an isomorphism for $d \leq \tfrac{\mathrm{rk}(M) - 4}{2}$ and an epimorphism for $d \leq \tfrac{\mathrm{rk}(M) - 3}{2}$ (with an improvement if $M$ is a free module). One deduces the same stability range for the maps \eqref{eq:GeneralStab} by exploiting the operation of tensoring with rank 1 projective modules.

\subsection{The homological stability statement}

We will prove a generic homological stability theorem in the context of $E_2$-algebras, having the following explicit consequence for general linear groups over Dedekind domains.

\begin{thm}\label{thm:StabDedekind}
Let $\cO$ be a Dedekind domain, $[L] \in \mathrm{Pic}(\cO)$, and $M$ be a finitely-generated projective $\cO$-module. Then
$$(- \oplus \mathrm{Id}_L)_* : H_d(\GL(M) ) \lra H_d(\GL(M \oplus L))$$
is an isomorphism for $d < \tfrac{\mathrm{rk}(M)-2}{2}$ and an epimorphism for $d < \tfrac{\mathrm{rk}(M)}{2}$.
\end{thm}

Given the selection of stabilisation maps, one for each element of $\mathrm{Pic}(\cO)$, it is a little unnatural to formulate a result such as Theorem \ref{thm:StabDedekind} for one stabilisation map at a time. Instead, the symmetric monoidal groupoid of all finitely-generated projective $\cO$-modules has homology
$$\bigoplus_{[M]} H_*(\GL(M) ),$$
where the sum is over representatives of each isomorphism class. Direct sum of modules endows this with the structure of a graded-commutative algebra, and recording the rank of a module endows it with an additional $\bN$-grading. By the classification of finitely-generated projective $\cO$-modules, in homological degree 0 we have an isomorphism of $\bN$-graded algebras
$$\bigoplus_{[M]} H_0(\GL(M) ) \cong  \bZ[\mathrm{Pic}(\cO)]/(\rho \cdot \rho' - \sigma \cdot (\rho\otimes\rho')) =: A_\cO$$
where we write $(\mathrm{Pic}(\cO), \otimes, \sigma=[\cO])$ for the abelian group of $\otimes$-invertible $\cO$-modules and $A_\cO$ is $\bN$-graded by considering each element of $\mathrm{Pic}(\cO)$ to have grading 1.\footnote{Here $\bZ[\mathrm{Pic}(\cO)]$ denotes the polynomial algebra on the set $\mathrm{Pic}(\cO)$, not the group algebra!} This discussion shows that for each homological degree $d$,
$$\bigoplus_{[M]} H_d(\GL(M) )$$
has the structure of an $\bN$-graded $A_\cO$-module. Our more conceptual homological stability result in this setting is then as follows.

\begin{thm}\label{thm:StabDedekindRegularity}
For each degree $d$, the $\bN$-graded $A_\cO$-module $\bigoplus_{[M]} H_d(\GL(M))$ is generated in gradings $\leq 2d$, and presented in gradings $\leq 2d+1$.
\end{thm}

In Section \ref{sec:Example} we will justify the following explicit example.

\begin{example}
The Dedekind domain $\cO := \bZ[\sqrt{-5}]$ has $\mathrm{Pic}(\cO) = \{\sigma, \lambda\}$ with $\sigma = [\cO]$ and nontrivial element $\lambda$ represented by the ideal $\mathfrak{l}=(2,1+\sqrt{-5})$. The abelianisations of the automorphism groups of the finitely-generated projective $\cO$-modules are:
\begin{center}
\begin{tabular}{ c|cccc } 
 $n$ &  1 & 2 & 3 & 4 \\ \hline
 $\GL(\cO^n)^{ab}$ &  $\bZ/2\{X\}$ & $\bZ/2\{\sigma \cdot X,U,T,B,C\}$ & $\bZ/2\{\sigma^2 \cdot X\}$ & $\bZ/2\{\sigma^3 \cdot X\}$\\ 
 $\GL(\cO^{n-1} \oplus \mathfrak{l})^{ab}$ &  $\bZ/2\{X'\}$ & $\bZ/2\{\sigma \cdot X', C', D'\}$ & $\bZ/2\{\sigma^2 \cdot X'\}$ & $\bZ/2\{\sigma^3 \cdot X'\}$
\end{tabular}
\end{center}
where $\lambda \cdot X = \sigma \cdot X'$, $\lambda \cdot X' = \sigma \cdot X$, and $U, T, B, C, C', D'$ are all annihilated by $\sigma$ and $\lambda$. We see that this $A_\cO$-module is indeed generated in gradings $\leq 2$ and presented in gradings $\leq 3$.
\end{example}

It is a simple exercise to deduce Theorem \ref{thm:StabDedekind} from Theorem \ref{thm:StabDedekindRegularity}, but Theorem \ref{thm:StabDedekindRegularity} is stronger. For example, it includes the statement that for each $[L] \in \mathrm{Pic}(\cO)$ the sum of the stabilisation maps
$$\bigoplus_{[L'] \in \mathrm{Pic}(\cO)} (- \oplus L \otimes (L')^{-1})_* : \bigoplus_{[L'] \in \mathrm{Pic}(\cO)} H_d(\GL(\cO^{n-2} \oplus L') ) \lra H_d(\GL(\cO^{n-1} \oplus L) )$$
is surjective as long as $n > 2d$ (whereas Theorem \ref{thm:StabDedekind} says that the individual stabilisation maps are surjective as long as $n+1 > 2d$).

\subsection{Reductive Borel--Serre spaces}
The spaces $B\GL(M)$ can be considered as an unstable approximation to the algebraic $K$-theory space $\Omega^\infty_0\mathrm{K}(\cO)$, from which point of view the homological stability results of the previous section can be viewed as estimating the quality of this approximation.

Another unstable approximation to $\Omega^\infty_0\mathrm{K}(\cO)$ has been recently introduced by Clausen and Jansen \cite{ClausenJansen}, in the form of the reductive Borel--Serre spaces $|\RBS(M)|$; they are better approximations in the sense that there is a factorisation
$$B\GL(M) \lra |\RBS(M)| \lra \Omega^\infty_0\mathrm{K}(\cO).$$
Using recent work of Jansen \cite{Jansen2}, we will explain how our general method applies equally well to these spaces, and the exact analogues of Theorems \ref{thm:StabDedekind} and \ref{thm:StabDedekindRegularity} hold with $|\RBS(M)|$ in place of $B\GL(M)$. 

\begin{rem}
We will show elsewhere that when $\cO$ has class number 1 then the $|\RBS(M)|$ enjoy homological stability with slope 1, which is much better than the $B\GL(M)$. Perhaps this is also the case for general $\cO$; we do not pursue that here.
\end{rem}

\subsection{The generic homological stability statement}

To formulate our generic homological stability theorem, we assume familiarity with the language and notation developed in \cite{e2cellsI}. Fix a commutative ring $\bk$, and work in the category $\mathsf{sMod}_\bk^\bN$ of $\bN$-graded simplicial $\bk$-modules. The statement will make use of the following construction, which axiomatises what we did in the previous section. Let $P = (P, \star, \sigma)$ be an abelian group, and define a commutative $\bk$-algebra
$$A_P := \bk[P]/(\rho \cdot \rho' - \sigma \cdot (\rho\star \rho') \, | \, \rho, \rho' \in P).$$
As the relations imposed are homogeneous with respect to word-length in the free commutative $\bk$-algebra $\bk[P]$, we may consider $A_P$ as being $\bN$-graded. It has an augmentation $\epsilon : A_P \to \bk$ by sending all $\rho \in P$ to 0. Following \cite[\S 12.2.1]{e2cellsI}, for a non-unital $E_2$-algebra $\gR$ we write $\overline{\gR}$ for the unital associative algebra obtained by adding a unit and strictifying the $E_1$-structure. 

\begin{thm}\label{thm:StabGeneric}
Let $\gR \in \mathsf{Alg}_{E_2}(\mathsf{sMod}_\bk^\bN)$ satisfy $H_{0,0}(\gR)=0$ as well as
\begin{enumerate}[(I)]
\item $H_{n,d}^{E_1}(\gR)=0$ for $d < n-1$, and
\item there is an abelian group $(P, \star, \sigma)$ such that $H_{*,0}(\overline{\gR}) \cong A_P$ as $\bN$-graded $\bk$-algebras.
\end{enumerate}
Then for each degree $d$, $H_{*,d}(\overline{\gR})$ has the structure of an $\bN$-graded $A_P$-module, and as such it is generated in gradings $\leq 3d$ and presented in gradings $\leq 3d+1$. In particular, for any element $\rho \in P \subset H_{1,0}(\overline{\gR})$ the homotopy cofibre $\overline{\gR}/\rho$ of the map $\rho \cdot - : S^{1,0} \otimes \overline{\gR} \to \overline{\gR}$ satisfies $H_{n,d}(\overline{\gR}/\rho)=0$ for $d < \tfrac{n-1}{3}$.

Suppose that in addition to the above we have
\begin{enumerate}[(I)]
\setcounter{enumi}{2}
\item $H_{3,2}^{\overline{\gR}}(A_P)=0$.
\end{enumerate}
Then the $\bN$-graded  $A_P$-module $H_{*,d}(\overline{\gR})$ is generated in gradings $\leq 3d-1$, and is presented in gradings $\leq 3d+1$, but also $H_{n,d}(\overline{\gR}/\rho)=0$ for $d < \tfrac{n-1}{2}$.

Suppose that in addition to all the above we have
\begin{enumerate}[(I)]
\setcounter{enumi}{3}
\item $H_{4,3}^{\overline{\gR}}(A_P)=0$.
\end{enumerate}
Then the $\bN$-graded  $A_P$-module $H_{*,d}(\overline{\gR})$ is generated in gradings $\leq 2d$, and is presented in gradings $\leq 2d+1$. 
\end{thm}

\begin{rem}
The special case $P=\bZ/2$ has been analysed in much more detail by Sierra \cite[Section 2]{SierraQuad}. See also \cite[Theorem 2.4]{SierraMfld}.
\end{rem}

We also record here the following theorem, which resulted from early attempts to prove a stability range of slope $\tfrac{1}{2}$ only under assumptions (I) and (II). It is mainly notable because the stability range $d < \tfrac{n-\sqrt{n}}{2}$ in the first part is best possible under assumptions (I) and (II) (see Remark \ref{rem:sqrtEstimateSharp}). We find this surprising, as we have never seen an honestly non-linear stability range ``in nature''.

\begin{thm}\label{thm:RationalStab}
Let $\gR \in \mathsf{Alg}_{E_\infty}(\mathsf{sMod}_\bk^\bN)$ satisfy (I) and (II) above, with $\bk$ a field of characteristic zero. Then we have $H_{n,d}(\overline{\gR}/\rho)  =0$ for $d < \tfrac{n-\sqrt{n}}{2}$.

If $P$ is finite then we also have $H_{n,d}(\overline{\gR}/\rho)=0$ for $d < \tfrac{n-|P|+1}{2}$.
\end{thm}

\vspace{2ex}

\noindent \textbf{Acknowledgements.} I would like to thank Mikala Ørsnes Jansen for sharing a draft of \cite{Jansen2}, and Robin J.\ Sroka for a useful discussion around the subject of Theorem \ref{thm:SteinbergEpi}. I was supported by the ERC under the European Union’s Horizon 2020 research and innovation programme (grant agreement No.~756444).

\section{Proof of Theorem \ref{thm:StabGeneric}}

\subsection{Regularity for $A_P$}

We begin with some homological algebra over $A := A_P$, considered with its $\bN$-grading, and with its unique augmentation $\epsilon : A \to \bk$. We write $I := I_P$ for its augmentation ideal. From the relations defining $A_P$ we see that: $A(0) = \bk$, $A(1) = \bk\{P\}$, and more generally for any $n \geq 1$ there is an isomorphism $\bk\{P\} \overset{\sim}\to A(n)$ given by sending $\rho \in P$ to $\sigma^{\cdot(n-1)} \cdot \rho$.

For an $\bN$-graded $A$-module $M$, we will consider the (bigraded) $\mathrm{Tor}$-groups 
$$H^{A}_{n,d}(M) := H_{n,d}(\bk \otimes^\bL_A M) =  \mathrm{Tor}^{A}_{n, d}(\bk, M),$$ 
and write
$$h^{A}_d(M) := \max\{n \in \bN \, | \, H^{A}_{n,d}(M) \neq 0\}.$$
The module $M$ is therefore generated in gradings $\leq h_0^{A}(M)$, and is presented in gradings $\leq \max(h_0^{A}(M),h_1^{A}(M))$. 

\begin{lem}\label{lem:RegA}
We have $h^A_d(M) \leq d-1+ h_1^A(M)$ for $d \geq 2$.
\end{lem}
\begin{proof}
Let $B := \bk[\sigma]$, a polynomial algebra in one variable, and let $B \to A$ send $\sigma$ to the symbol of the same name. Via this map we can consider $M$ as a $B$-module. There is an equivalence $\bk \otimes^\bL_B M \simeq (\bk \otimes_B^\bL A) \otimes^\bL_A M$.

We may compute $B$-module homology by using the free $B$-module resolution $0 \to S^{1,0} \otimes B \overset{\sigma \cdot -}\to B \to \bk \to 0$. This shows that $H^B_{*,d}(-)=0$ for $d \geq 2$, and with the explicit presentation of $A$ gives
$$H^B_{n,d}(A) = \begin{cases}
\bk & (n,d) = (0,0)\\
\bk\{P\}/\bk\{\sigma\} & (n,d) = (1,0)\\
0 & \text{else}.
\end{cases}$$
In particular there is a homotopy cofibre sequence of right $A$-modules of the form
$$S^{1,0} \otimes \bk\{P\}/\bk\{\sigma\} \lra \bk \otimes_B^\bL A \lra S^{0,0} \otimes \bk.$$
Applying $- \otimes^\bL_A M$ and using the equivalence above leads to a long exact sequence
\begin{equation*}
\begin{tikzcd}
 \cdots (\bk\{P\}/\bk\{\sigma\}) \otimes H^A_{n-1,d}(M) \rar & H^B_{n,d}(M) \rar
             \ar[draw=none]{d}[name=Z, anchor=center]{}
    & H^A_{n,d}(M) \ar[rounded corners,
            to path={ -- ([xshift=2ex]\tikztostart.east)
                      |- (Z.center) \tikztonodes
                      -| ([xshift=-2ex]\tikztotarget.west)
                      -- (\tikztotarget)}]{dll}[at end, swap]{\partial} \\ 
  (\bk\{P\}/\bk\{\sigma\}) \otimes H^A_{n-1,d-1}(M) \rar & H^B_{n,d-1}(M) \rar
    & H^A_{n,d-1}(M) \cdots.
\end{tikzcd}
\end{equation*}
For $d \geq 2$, as $H^B_{*,d}(M)=0$ this gives an injection
$$\partial : H^A_{n,d}(M) \lra (\bk\{P\}/\bk\{\sigma\}) \otimes H^A_{n-1,d-1}(M)$$
so that $h^A_d(M) \leq h^A_{d-1}(M)+1$ as required.
\end{proof}

\begin{cor}\label{cor:AIsKoszul}
The algebra $A$ is Koszul, i.e.\ $\mathrm{Tor}^A_{n,d}(\bk, \bk)=0$ for $n \neq d$.
\end{cor}
\begin{proof}
Applying Lemma \ref{lem:RegA} with $M=\bk$, which has $h_1^A(\bk)=1$, shows that $\mathrm{Tor}^A_{n,d}(\bk, \bk)=0$ for $d < n$. On the other hand, using the reduced bar complex to calculate these $\mathrm{Tor}$-groups shows that $\mathrm{Tor}^A_{*,d}(\bk, \bk)$ is a subquotient of $I^{\otimes d}$, for $I = \mathrm{Ker}(\epsilon : A \to \bk)$ the augmentation ideal. As $I$ is supported in gradings $\geq 1$, $I^{\otimes d}$ is supported in gradings $\geq d$, so $\mathrm{Tor}^A_{n,d}(\bk, \bk)=0$ for $n < d$.
\end{proof}

\subsection{Regularity for $\overline{\gR}$}

Write $\gA$ for $A$ considered as a discrete unital $E_2$-algebra in $\mathsf{sMod}_\bk^\bN$, and $\gI$ for its augmentation ideal. We can consider an $A$-module $M$ as an $\overline{\gR}$-module $\gM$, via the truncation map $\overline{\gR} \to \tau_{\leq 0} \overline{\gR} \simeq \gA$, and we can therefore define $\overline{\gR}$-module homology groups $H^{\overline{\gR}}_{n,d}(\gM) := H_{n,d}(\bk \otimes_{\overline{\gR}}^\bL \gM)$ as well as
$$h^{\overline{\gR}}_d(\gM) := \max\{n \in \bN \, | \, H^{\overline{\gR}}_{n,d}(\gM) \neq 0\}.$$

\begin{lem}\label{lem:AssIAlt}
Assumption (I) is equivalent to ``$H^{\overline{\gR}}_{n,d}(\bk)=0$ for $d < n$''.
\end{lem}
\begin{proof}
Use the equivalences $B^{E_1}({\gR}^+,\epsilon) \simeq \bk \oplus \tilde{B}^{E_1}({\gR})$ from \cite[Lemma 13.5]{e2cellsI} and $\tilde{B}^{E_1}({\gR}) \simeq \Sigma Q^{E_1}_\bL(\gR)$ \cite[Theorem 13.7]{e2cellsI}, and note that there are equivalences
$$B^{E_1}({\gR}^+,\epsilon)\simeq B^{E_1}(\overline{\gR},\epsilon)\simeq B(\bk, \overline{\gR}, \bk)$$
where the first is induced by \cite[Proposition 12.9]{e2cellsI} and the second is induced by the semi-simplicial map $B^{E_1}_\bullet(\overline{\gR},\epsilon)\to B_\bullet(\bk, \overline{\gR}, \bk)$ induced on $p_1$-simplices by the equivalence $\mathcal{P}_1(p_1) \overset{\sim}\to \{*\}$.
\end{proof}

\begin{lem}\label{lem:RHomOfA}
We have $H^{\overline{\gR}}_{n,d}(\gA)=0$ for $d < n-1$, and for $d < 2$ with the exception $H^{\overline{\gR}}_{0,0}(A)=\bk$.
\end{lem}
\begin{proof}
Consider the homotopy cofibre sequence of left $\gA$-modules, and hence of $\overline{\gR}$-modules,
$\gA \otimes S^{1,0} \overset{- \cdot \sigma}\to \gA \to \gA/\sigma$. As $\gA$ is $E_2$ (and in fact commutative), the left-hand map is nullhomotopic on $\overline{\gR}$-module indecomposables, giving short exact sequences
$$0 \lra H^{\overline{\gR}}_{n,d}(\gA) \lra H^{\overline{\gR}}_{n,d}(\gA/\sigma) \lra H^{\overline{\gR}}_{n-1,d-1}(\gA) \lra 0.$$
As in the proof of Lemma \ref{lem:RegA} there is a homotopy cofibre sequence of left $\overline{\gR}$-modules $S^{1,0} \otimes \bk\{P\}/\bk\{\sigma\} \to \gA/\sigma \to S^{0,0} \otimes \bk$, where the module structure on the outer terms is via the augmentation. Together with the fact that $H^{\overline{\gR}}_{n,d}(\bk)=0$ for $d < n$, which follows from assumption (I) via Lemma \ref{lem:AssIAlt}, it follows that $H^{\overline{\gR}}_{n,d}(\gA/\sigma)=0$ for $d < n-1$, and so the same holds for $H^{\overline{\gR}}_{n,d}(\gA)$.

The second claim follows from the fact that $\overline{\gR} \to \gA$ is 1-connected in each grading, as it is the 0-truncation, so the Hurewicz theorem \cite[Corollary 11.12]{e2cellsI} shows that $H^{\overline{\gR}}_{n,d}(\gA, \overline{\gR})=0$ for $d \leq 1$.
\end{proof}

\begin{lem}\label{lem:RModREgularityEstimates} 
Let $\gM$ be an $\overline{\gR}$-module which is discrete.
\begin{enumerate}[(i)]
\item We have $h^{\overline{\gR}}_2(\gM) \leq \max\{3+h^{\overline{\gR}}_0(\gM), 1+h^{\overline{\gR}}_1(\gM)\}$, and for $d \geq 3$ we have $h^{\overline{\gR}}_d(\gM) \leq d+1 + \max\{h^{\overline{\gR}}_0(\gM), h^{\overline{\gR}}_1(\gM)-1\}$.

\item If $H^{\overline{\gR}}_{3,2}(\gA)=0$, then we have $h^{\overline{\gR}}_2(\gM) \leq 2+\max\{h^{\overline{\gR}}_0(\gM), h^{\overline{\gR}}_1(\gM)-1\}$.

\item In in addition $H^{\overline{\gR}}_{4,3}(\gA)=0$, then we have $h^{\overline{\gR}}_d(\gM) \leq d+\max\{h^{\overline{\gR}}_0(\gM), h^{\overline{\gR}}_1(\gM)-1\}$ for $d \leq 3$.
\end{enumerate}
\end{lem}
\begin{proof}
As $\gM$ is discrete, the $\overline{\gR}$-module structure on it is restricted along $\overline{\gR} \to \gA$. Thus we have an equivalence
$$\bk \otimes_{\overline{\gR}}^\bL \gM \simeq (\bk \otimes_{\overline{\gR}}^\bL \gA) \otimes^\bL_\gA \gM.$$
We may filter the right $\gA$-module $\bk \otimes_{\overline{\gR}}^\bL \gA$ by its grading: the associated graded may be identified with $\bk \otimes_{\overline{\gR}}^\bL \gA$, but the $\gA$-module structure is now via the augmentation $\epsilon: \gA \to \bk$. It gives a spectral sequence
$$E^1_{n,p,q} = H_p((\bk \otimes_{\overline{\gR}}^\bL \gA)(q) \otimes^\bL_\bk (\bk \otimes_\gA^\bL \gM)(n-q)) \Longrightarrow H_{n,p}(\bk \otimes^\bL_{\overline{\gR}} \gM) = H^{\overline{\gR}}_{n,p}(\gM).$$

By the second part of Lemma \ref{lem:RHomOfA} we have that $H^{\overline{\gR}}_{n,d}(\gA)$ vanishes for $d \leq 1$ except that it is $\bk$ in bidgree $(0,0)$. From the K{\"u}nneth theorem we then see that
\begin{equation}\label{eq:RHomologyIsAHomology}
H^{\overline{\gR}}_{n,0}(\gM) \cong H^{\gA}_{n,0}(\gM)\quad\text{and}\quad H^{\overline{\gR}}_{n,1}(\gM) \cong H^{\gA}_{n,1}(\gM),
\end{equation}
so $h^{\overline{\gR}}_d(\gM)=h^{\gA}_d(\gM)$ for $d \leq 1$. More generally, we may approximate the $E^1$-page using the K{\"u}nneth spectral sequence
$$F^2_{u,v} = \bigoplus_{v'+v'' = v} \mathrm{Tor}^\bk_{u}(H^{\overline{\gR}}_{q,v'}(\gA), H^\gA_{n-q, v''}(\gM)) \Longrightarrow E^1_{n,u+v,q}.$$ 
Thus we find that
$$h^{\overline{\gR}}_{d}(\gM) \leq \max_{v'+v''\leq d}\{h^{\overline{\gR}}_{v'}(\gA) + h^{\gA}_{v''}(\gM)\}.$$
Using that
\begin{enumerate}[(i)]
\item $h^\gA_{v''}(\gM) \leq v''-1 + h_1^\gA(\gM) = v''-1 + h^{\overline{\gR}}_1(\gM)$ for $v'' \geq 2$, by Lemma \ref{lem:RegA}, and
\item $h^{\overline{\gR}}_{0}(\gA)=0$, $h^{\overline{\gR}}_{1}(\gA)=-\infty$, and $h^{\overline{\gR}}_{v'}(\gA) \leq v'+1$ for $v' \geq 2$, by Lemma \ref{lem:RHomOfA},
\end{enumerate}
we find the claimed estimates.

Under the additional assumption $h^{\overline{\gR}}_{2}(\gA) \leq 2$ we find the slightly improved second estimate. Adding the further assumption $h^{\overline{\gR}}_{3}(\gA) \leq 3$ gives the third estimate.
\end{proof}

The following can be used to verify the hypotheses in Lemma \ref{lem:RModREgularityEstimates} (ii) and (iii).

\begin{lem}\label{lem:CheckHyp3and4}\mbox{}
\begin{enumerate}[(i)]
\item If the map $H^{\overline{\gR}}_{3,3}(\bk) \to H^{\gA}_{3,3}(\bk)$ is surjective then $H_{3,2}^{\overline{\gR}}(\gA)=0$.

\item If in addition the map $H^{\overline{\gR}}_{4,4}(\bk) \to H^{\gA}_{4,4}(\bk)$ is surjective then $H_{4,3}^{\overline{\gR}}(\gA)=0$.

\end{enumerate}
\end{lem}
\begin{proof}
Truncation gives a map of non-unital $E_2$-algebras $\gR \to \gI$, which is 1-connected in each grading, i.e.\ $\gI$ can be obtained from $\gR$ by attaching ($E_1$-)cells of dimension $\geq 2$, and so $H^{E_1}_{n,d}(\gI, \gR)=0$ for $d < 2$. In addition $H^{E_1}_{n,d}(\gR)=0$ for $d < n-1$ by assumption.

Recall from the proof of Lemma \ref{lem:AssIAlt} that $H^{\overline{\gR}}_{n,d}(\bk) \cong H_{n,d-1}^{E_1}(\gR)$ for $n>0$, and by the same argument $H^{\gA}_{n,d}(\bk) \cong H_{n,d-1}^{E_1}(\gI)$. From the latter and Corollary \ref{cor:AIsKoszul} it follows that $H^{E_1}_{n,d}(\gI)=0$ for $d < n-1$ (in fact for $d \neq n-1$). The long exact sequence for the pair $(\gI, \gR)$ on $E_1$-homology therefore shows that $H^{E_1}_{n,d}(\gI, \gR)=0$ for $d < n-1$. 

In addition this sequence contains the portion
$$H_{3,2}^{E_1}(\gR) \lra H_{3,2}^{E_1}(\gI) \lra H_{3,2}^{E_1}(\gI, \gR) \lra H_{3,1}^{E_1}(\gR)=0.$$
The assumption in part (i) of this lemma is that the left-hand map is surjective, showing that $H_{3,2}^{E_1}(\gI, \gR)=0$. Similarly there is a portion 
$$H_{4,3}^{E_1}(\gR) \lra H_{4,3}^{E_1}(\gI) \lra H_{4,3}^{E_1}(\gI, \gR) \lra H_{4,2}^{E_1}(\gR)=0$$
and the assumption in part (ii) of this lemma is that the left-hand map is surjective, showing that $H_{4,3}^{E_1}(\gI, \gR)=0$.

As $E_2$-homology may be calculated from $E_1$-homology by a bar construction (see \cite[Theorem 14.4]{e2cellsI}) we find that $H^{E_1}_{n,d}(\gR) \to H^{E_2}_{n,d}(\gR)$ is surjective for $d \leq n-1$, and similarly for $\gI$, from which it follows that $H^{E_2}_{n,d}(\gI, \gR)=0$ for $d < n-1$, and that on the line $d=n-1$ there is an epimorphism
$$H^{E_1}_{n,n-1}(\gI, \gR) \lra H^{E_2}_{n,n-1}(\gI, \gR).$$
In particular $H_{2,1}^{E_2}(\gI, \gR)$ vanishes, and $H_{3,2}^{E_2}(\gI, \gR)$ and $H_{4,3}^{E_2}(\gI, \gR)$ also vanish under the assumptions in parts (i) and (ii).

We apply \cite[Theorem 15.9]{e2cellsI} with $\rho(n)=n$ and 
$$\sigma(n) = \begin{cases}
1 & n=1\\
2 & n=2\\
3 & n=3 \text{ and the hypothesis of (i) holds}\\
4 & n=4 \text{ and the hypothesis of (ii) holds}\\
n-1 & \text{else},
\end{cases}$$
showing that there is a map
$$H^{\overline{\gR}}_{n,d}(\gA, \overline{\gR}) \lra H^{E_2}_{n,d}(\gI, \gR)$$
which is an isomorphism for $d < (\sigma *\sigma)(n)$. If the hypothesis of (i) holds then $(\sigma * \sigma)(3) = 3$ and so $H^{\overline{\gR}}_{3,2}(\gA) = H^{\overline{\gR}}_{3,2}(\gA, \overline{\gR}) \overset{\sim}\to H^{E_2}_{3,2}(\gI, \gR)=0$. If in addition the hypothesis of (ii) holds then $(\sigma*\sigma)(4)=4$ and so $H^{\overline{\gR}}_{4,3}(\gA) = H^{\overline{\gR}}_{4,3}(\gA, \overline{\gR}) \overset{\sim}\to H^{E_2}_{4,3}(\gI, \gR)=0$.
\end{proof}

\subsection{The $A$-module part of Theorem \ref{thm:StabGeneric}}


We consider $\bk \simeq \bk \otimes^\bL_{\overline{\gR}} \overline{\gR}$, and give the left $\overline{\gR}$-module $\overline{\gR}$ its Postnikov filtration. This leads to a spectral sequence
$$E^2_{n,s,t} = H_{n,s}^{\overline{\gR}}(H_{*,t}(\overline{\gR})) \Longrightarrow H_{n,t+s}(\bk)$$
converging very nearly to zero, with differentials $d^r : E^r_{n,s,t} \to E^r_{n, s-r, t+r-1}$. We have $H_{*,0}(\overline{\gR})=A$, and so $E^1_{*,s,0}=H^{\overline{\gR}}_{*,s}(\gA)$, which is supported in degrees $\leq s+1$ by Lemma \ref{lem:RHomOfA}. Furthermore, it vanishes for $s=1$ and is $\bk$ for $s=0$. We consider the following chart depicting (upper bounds for) the supports of the terms in this spectral sequence.

\begin{figure}[h]
\begin{sseqdata}[ name = BarSS, homological Serre grading, classes = {draw = none }, x range = {0}{4}, y range = {0}{1}, xscale = 2, yscale=0.9, y axis gap = 35pt]

\class["0"](0,0)
\class["-\infty"](1,0)
\class["\leq 3"](2,0)
\class["\leq 4"](3,0)
\class["\leq 5"](4,0)
\class["\leq 6"](5,0)

\class["{h_0^{\overline{\gR}}(H_{*,1}({\overline{\gR}}))}"](0,1)
\class["{h_1^{\overline{\gR}}(H_{*,1}({\overline{\gR}}))}"](1,1)

\d2(2,0)
\d2(3,0)

\end{sseqdata}
\printpage[ name = BarSS, page = 2 ]
\end{figure}

As the spectral sequence converges to zero in positive total degrees, the two indicated $d^2$-differentials must be epimorphisms. This shows that $h_0^{\overline{\gR}}(H_{*,1}({\overline{\gR}})) \leq 3$ and $h_1^{\overline{\gR}}(H_{*,1}({\overline{\gR}})) \leq 4$, so by Lemma \ref{lem:RModREgularityEstimates} (i) we also have
$$h_s^{\overline{\gR}}(H_{*,1}({\overline{\gR}})) \leq 4+s \text{ for } s \geq 2.$$
This allows us to fill in the 1-row. Now we consider the differentials entering the first two positions of the 2-row:

\begin{figure}[h]
\begin{sseqdata}[ name = BarSS2, homological Serre grading, classes = {draw = none }, x range = {0}{4}, y range = {0}{2}, xscale = 2, yscale=0.9, y axis gap = 35pt]

\class["0"](0,0)
\class["-\infty"](1,0)
\class["\leq 3"](2,0)
\class["\leq 4"](3,0)
\class["\leq 5"](4,0)
\class["\leq 6"](5,0)

\class["\leq 3"](0,1)
\class["\leq 4"](1,1)
\class["\leq 6"](2,1)
\class["\leq 7"](3,1)
\class["\leq 8"](4,1)

\class["{h_0^{\overline{\gR}}(H_{*,2}({\overline{\gR}}))}"](0,2)
\class["{h_1^{\overline{\gR}}(H_{*,2}({\overline{\gR}}))}"](1,2)

\end{sseqdata}
\printpage[ name = BarSS2, page = 2---0 ]
\end{figure}

By considering all the differentials that can arrive at these groups, we see that $h_0^{\overline{\gR}}(H_{*,2}({\overline{\gR}})) \leq 6$ and $h_1^{\overline{\gR}}(H_{*,2}({\overline{\gR}})) \leq 7$, and Lemma \ref{lem:RModREgularityEstimates} (i) allows us to again complete this row by $h_s^{\overline{\gR}}(H_{*,2}({\overline{\gR}})) \leq 7+s$ for $s \geq 2$. In this way, we see that
$$
{h_0^{\overline{\gR}}(H_{*,t}({\overline{\gR}}))} \leq 3t, \quad
{h_1^{\overline{\gR}}(H_{*,t}({\overline{\gR}}))}  \leq 3t+1, \quad  {h_s^{\overline{\gR}}(H_{*,t}({\overline{\gR}}))}  \leq 3t+1+s \text{ for } s \geq 2.
$$
Using \eqref{eq:RHomologyIsAHomology} the same follows for $h_i^A(H_{*,d}({\overline{\gR}}))$. This proves the basic case.

Suppose now that $H_{3,2}^{\overline{\gR}}(\gA)=0$. This allows us to improve the start of the induction, as $E^1_{*,2,0}$ is supported in degrees $\leq 2$, and also allows us to use Lemma \ref{lem:RModREgularityEstimates} (ii). The same line of reasoning as above then leads to
\begin{align*}
{h_0^{A}(H_{*,t}({\overline{\gR}}))} \leq 3t-1, &\quad\quad
{h_1^{A}(H_{*,t}({\overline{\gR}}))}  \leq 3t+1, \\
{h_2^{A}(H_{*,t}({\overline{\gR}}))}  \leq 3t+2, & \quad\quad {h_s^{A}(H_{*,t}({\overline{\gR}}))}  \leq 3t+1+s \text{ for } s \geq 3,
\end{align*}
depicted as follows.

\begin{figure}[h]
\begin{sseqdata}[ name = BarSS3, homological Serre grading, classes = {draw = none }, x range = {0}{4}, y range = {0}{3}, xscale = 2, yscale=0.9, y axis gap = 35pt]

\class["0"](0,0)
\class["-\infty"](1,0)
\class["\leq 2"](2,0)
\class["\leq 4"](3,0)
\class["\leq 5"](4,0)
\class["\leq 6"](5,0)

\class["\leq 2"](0,1)
\class["\leq 4"](1,1)
\class["\leq 5"](2,1)
\class["\leq 7"](3,1)
\class["\leq 8"](4,1)

\class["\leq 5"](0,2)
\class["\leq 7"](1,2)
\class["\leq 8"](2,2)
\class["\leq 10"](3,2)
\class["\leq 11"](4,2)

\class["\leq 8"](0,3)
\class["\leq 10"](1,3)
\class["\leq 11"](2,3)
\class["\leq 13"](3,3)
\class["\leq 14"](4,3)

\end{sseqdata}
\printpage[ name = BarSS3, page = 2 ]
\end{figure}

Finally, suppose that in addition $H_{4,3}^{\overline{\gR}}(\gA)=0$. This allows us to further improve the start of the induction, as $E^1_{*,3,0}$ is supported in degrees $\leq 3$, and also allows us to use Lemma \ref{lem:RModREgularityEstimates} (iii). The same line of reasoning as above then leads to
\begin{equation*}
{h_s^{A}(H_{*,t}({\overline{\gR}}))} \leq 2t+s \text{ for } s \geq 1
\end{equation*}
depicted as follows.

\begin{figure}[h]
\begin{sseqdata}[ name = BarSS4, homological Serre grading, classes = {draw = none }, x range = {0}{4}, y range = {0}{3}, xscale = 2, yscale=0.9, y axis gap = 35pt]

\class["0"](0,0)
\class["-\infty"](1,0)
\class["\leq 2"](2,0)
\class["\leq 3"](3,0)
\class["\leq 5"](4,0)
\class["\leq 6"](5,0)

\class["\leq 2"](0,1)
\class["\leq 3"](1,1)
\class["\leq 4"](2,1)
\class["\leq 5"](3,1)
\class["\leq 7"](4,1)

\class["\leq 4"](0,2)
\class["\leq 5"](1,2)
\class["\leq 6"](2,2)
\class["\leq 7"](3,2)
\class["\leq 9"](4,2)

\class["\leq 6"](0,3)
\class["\leq 7"](1,3)
\class["\leq 8"](2,3)
\class["\leq 9"](3,3)
\class["\leq 11"](4,3)

\end{sseqdata}
\printpage[ name = BarSS4, page = 2 ]
\end{figure}

\subsection{The ordinary homological stability part of Theorem \ref{thm:StabGeneric}}

The remaining part of Theorem \ref{thm:StabGeneric} is the statement that $H_{n,d}(\overline{\gR}/\rho)=0$ for $d < \tfrac{n-1}{2}$ under assumptions (I), (II), and (III). (This cannot, of course, be deduced from the statement that the $A$-module $H_{*,d}(\gR)$ is generated in gradings $\leq 3d-1$ and presented in gradings $\leq 3d+1$, which we proved in the previous section under these assumptions.)

As in the proof of Lemma \ref{lem:CheckHyp3and4} we have $H_{n,d}^{E_2}(\gI, \gR)=0$ for $d < 2$, for $d < n-1$, and for $(n,d)=(3,2)$. Thus by \cite[Theorem 11.21]{e2cellsI} we can construct a minimal relative CW-approximation $\gR \to \gC \overset{\sim}\to \gI$, and thereby obtain a spectral sequence
$$E^1_{n,p,q} = H_{n,p+q,q}(\overline{\gR}/\rho[0] \otimes E_\infty^+(\bigoplus_{\alpha \in I} S^{n_\alpha, d_\alpha, d_\alpha})) \Longrightarrow H_{n,p+q}({\gA}/\rho)$$
with $d_\alpha \geq 2$, $d_\alpha \geq n_\alpha-1$, and $(n_\alpha, d_\alpha) \neq (3,2)$. The abutment vanishes for $p+q>0$ as well as for $n > 1$.

\begin{figure}[h]
\begin{sseqdata}[ name = cells, homological Serre grading, classes = {draw = none }, x range = {0}{6}, y range = {0}{5}, xscale = 1, yscale=0.7]

\class["*"](1,2)
\class["*"](2,2)

\class["*"](1,3)
\class["*"](2,3)
\class["*"](3,3)
\class["*"](4,3)

\class["*"](1,4)
\class["*"](2,4)
\class["*"](3,4)
\class["*"](4,4)
\class["*"](5,4)

\class["*"](1,5)
\class["*"](2,5)
\class["*"](3,5)
\class["*"](4,5)
\class["*"](5,5)
\class["*"](6,5)

\end{sseqdata}
\printpage[ name = cells, page = 2 ]
\caption{The $E_\infty$-cells of $E_\infty^+(\bigoplus_{\alpha \in I} S^{n_\alpha, d_\alpha, d_\alpha})$ are supported in the indicated bidegrees, with homological degree vertically and grading horizontally.}\label{fig:Supp}
\end{figure}

\begin{lem}\label{lem:VanRange}
 $H_{n'', d'', r}(E_\infty^+(\bigoplus_{\alpha \in I} S^{n_\alpha, d_\alpha, d_\alpha}))$ is non-trivial only if $(n'',d'',r)=(0,0,0)$ or if $d'' \geq r \geq 2$ and $d''-1 \geq \tfrac{1}{2} n''$.
\end{lem}
\begin{proof}
The object $X := \bigoplus_{\alpha \in I} S^{n_\alpha, d_\alpha, d_\alpha}$ has homology supported in tridegrees $(n'',d'',r)$ satisfying $d'' = r$, and $n''$ and $d'$ as in Figure \ref{fig:Supp}. In particular, it is supported in the region $R$ of tridegrees $(n'',d'',r)$ satisfying $d'' \geq r \geq 2$ and $d''-1 \geq \tfrac{1}{2}n''$. The region of tridegrees $R$ is closed under addition. As the homology of $X$ is free over $\bk$, by the K{\"u}nneth theorem it follows that the homology of $X^{\otimes k}$ is also supported in the region of tridegrees $R$. The region of tridegrees $R$ is also upwards-closed in the homological degree direction, so it follows from the homotopy orbits spectral sequence that the homology of $(X^{\otimes k})_{h\fS_k}$ is also supported in the region of tridegrees $R$. As $E_\infty^+(X) \simeq \bigoplus_{k \geq 0} (X^{\otimes k})_{h\fS_k}$, it follows that apart from the $k=0$ summand this has homology supported in the region of tridegrees $R$.
\end{proof}

We show that $H_{n,d}(\overline{\gR}/\rho)=0$ for $d < \tfrac{1}{2}(n-1)$ by induction on $d$: it holds for $d=0$, as $H_{*,0}(\overline{\gR}/\rho) = H_{*,0}({\gA}/\rho) = A/\rho$ is supported in gradings 0 and 1 by the explicit presentation of the algebra $A$. Note that $E^1_{n,d,0} = H_{n,d}(\overline{\gR}/\rho)$ and consider differentials
$$d^r : E^r_{n, d-r+1, r} \lra E^r_{n,d,0},$$
with $r >0$. There is a K{\"u}nneth spectral sequence
\begin{equation*}
\bigoplus_{s \geq 0 }\bigoplus_{\substack{n'+n''= n \\ d'+d'' = d+1-s}} \Tor^\bk_s(H_{n', d'}(\overline{\gR}/\rho) , H_{n'', d'', r}(E_\infty^+(\bigoplus_{\alpha \in I} S^{n_\alpha, d_\alpha, d_\alpha}))) \Rightarrow E^1_{n,d-r+1,r}.
\end{equation*}
Using Lemma \ref{lem:VanRange}, as $r >0$ we see that if $H_{n'', d'', r}(E_\infty^+(\bigoplus_{\alpha \in I} S^{n_\alpha, d_\alpha, d_\alpha}))$ is non-trivial then
$$d'- \tfrac{1}{2}(n'-1) = d+1-s-d'' - \tfrac{1}{2}(n-n''-1) = \underbrace{d-\tfrac{1}{2}(n-1)}_{<0} + \underbrace{(\tfrac{1}{2}n''+1-d'')}_{\leq 0}-s,$$
so $d' < \tfrac{1}{2}(n'-1)$. But then as $d' < d$ it follows by induction that $H_{n', d'}(\overline{\gR}/\rho)$ is trivial. Thus $E^1_{n,d-r+1,r}=0$ for $r>0$ when $d < \tfrac{1}{2}(n-1)$, so no non-zero differentials can enter $E^r_{n,d,0}$. As $H_{n,d}(\gA/\rho)=0$ for $d < \tfrac{1}{2}(n-1)$, it follows that $E^1_{n,d,0}=0$.

\section{Application to Dedekind domains: Proof of Theorem \ref{thm:StabDedekindRegularity}}

Theorem \ref{thm:StabDedekindRegularity} will be deduced from Theorem \ref{thm:StabGeneric} by taking the $E_\infty$-algebra $\mathbf{BGL}$, constructed in \cite[Section 18.2]{e2cellsI}, having
$$\mathbf{BGL}(n) = \bigoplus_{\substack{[M] \\ M \text{ f.g.\ projective} \\ \mathrm{rk}(M) = n}} \bk[B\GL(M)].$$
Axiom (I) of Theorem \ref{thm:StabGeneric} is already verified in \cite[Section 18.2]{e2cellsI}: it follows from the connectivity of Charney's ``split Tits building'' \cite[Theorem 1.1]{Charney}. Axiom (II) of Theorem \ref{thm:StabGeneric} holds with $P = \mathrm{Pic}(\cO)$: this is simply a reformulation of the classification of finitely-generated projective $\cO$-modules (which was the inspiration for this axiom). Verifying axioms (III) and (IV) is more involved, and requires preparation.

\subsection{The Tits and split Tits buildings}

\begin{defn}
Let $M$ be a projective $\cO$-module. Let $T(M)$ denote the poset consisting on nonzero proper direct summands $P$ of $M$, ordered by inclusion. Let $\widetilde{T}(M)$ denote the poset consisting of pairs $(P, Q)$ of nonzero submodules of $M$ such that $M = P \oplus Q$, with order relation $(P, Q) \leq (P', Q')$ when $P \leq P'$ and $Q' \leq Q$. 
\end{defn}

Writing $T(\mathrm{frac}(\cO) \otimes_\cO M)$ for the poset of nonzero proper vector subspaces of $\mathrm{frac}(\cO) \otimes_\cO M$, the map
$$\mathrm{frac}(\cO) \otimes_\cO - : T(M) \lra T(\mathrm{frac}(\cO) \otimes_\cO M)$$
is an isomorphism of posets, with inverse given by $V \mapsto V \cap M$. The Solomon--Tits theorem shows that $T(\mathrm{frac}(\cO) \otimes_\cO M)$ is Cohen--Macaulay of dimension $(\mathrm{rk}(M)-2)$, and in particular is homotopy equivalent to a wedge of $(\mathrm{rk}(M)-2)$-spheres. The \emph{Steinberg module} is its unique nontrivial reduced homology group 
$$\mathrm{St}(M) := \tilde{H}_{n-2}(T(M);\bZ).$$

Charney has shown \cite[Theorem 1.1]{Charney} that the same holds for $\widetilde{T}(M)$:

\begin{thm}[Charney]
If $M$ has rank $n$ then $\widetilde{T}(M)$ is homotopy equivalent to a wedge of $(n-2)$-spheres.
\end{thm}

The \emph{split Steinberg module} is its unique nontrivial reduced homology group, 
$$\widetilde{\mathrm{St}}(M) := \tilde{H}_{n-2}(\widetilde{T}(M);\bZ).$$
There is a map of posets $(P, Q) \mapsto Q : \widetilde{T}(M) \to {T}(M)^{op}$. 

\begin{thm}\label{thm:SteinbergEpi}
For $\mathrm{rk}(M) \leq 4$ the induced map $\widetilde{\mathrm{St}}(M) \to {\mathrm{St}}(M)$ is surjective.
\end{thm}

\begin{rem}
We only require this result in this range of ranks, so have not tried hard (i.e.\ have not succeeded) to prove it for general $M$.
\end{rem}

\begin{proof}[Proof of Theorem \ref{thm:SteinbergEpi}]
We apply the slight extension \cite[Theorem 4.1]{e2cellsIII} of Quillen's \cite[Theorem 9.1]{QuillenPoset} to the map of posets $f : \widetilde{T}(M) \to {T}(M)^{op}$ given by $f(P, Q) = Q$, using the height function $t(V) := \mathrm{rk}(V)-1$. We have that ${T}(M)^{op}$ is $(\mathrm{rk}(M)-2)$-spherical, and for each $V \in T(M)$, $[{T}(M)^{op}]_{> V} \cong T(V)^{op}$ is $(t(V)-1)$-spherical.

For each $V \in T(M)$, we have $f_{\leq V} = \{(P,Q) \in \widetilde{T}(M) \, | \, V \subseteq Q\}$. Following \cite{e2cellsIII} we write $\mathcal{S}^{E_1}(\cdot, V \subseteq \cdot \, |\, M)$ for this poset. In order to apply \cite[Theorem 4.1]{e2cellsIII} we require it to be $(\mathrm{rk}(M)- \mathrm{rk}(V)-1)$-spherical. For $\mathrm{rk}(M) \leq 4$ there are three cases.

If $\mathrm{rk}(V) = \mathrm{rk}(M)-1$ then $\mathcal{S}^{E_1}(\cdot, V \subseteq \cdot \, |\, M)$ is the discrete poset consisting of pairs $(L,V)$ with $L \oplus V=M$. In particular it is 0-spherical.

If $\mathrm{rk}(V)=1$ then $\mathcal{S}^{E_1}(\cdot, V \subseteq \cdot \, |\, M)$ agrees with the complex called $[M\,|\,V]$ by Charney \cite{Charney}, so is $(\mathrm{rk}(M)-2)$-spherical by \cite[Theorem 1.1]{Charney}.

As long as $\mathrm{rk}(M) \leq 4$ the only remaining case is $\mathrm{rk}(M)=4$ and $\mathrm{rk}(V)=2$, in which case we require $\mathcal{S}^{E_1}(\cdot, V \subseteq \cdot \, |\, M)$ to be 1-spherical. It is clearly 1-dimensional, so we must show that it is 0-connected. Every element of this poset either has the form $(U, V)$ or the form $(L, W)$ for $W \supseteq V$ and $L$ of rank 1. In the second case, writing $W = L' \oplus V$, we have $(L, W) \leq (L \oplus L', V)$. So it suffices to show that any $(U, V)$ can be connected by a path to some fixed $(U_0, V)$. We will deduce this from the results of Charney used above. Let $L_0 \leq V$ be a rank 1 direct summand, and choose a complement $L_1$ in $V$ so that $U \oplus L_0 \oplus L_1 = M$. There are maps of posets
\begin{equation*}
\begin{tikzcd}
\mathcal{S}^{E_1}(\cdot, L_1 \subseteq \cdot \, |\, U \oplus L_1) \arrow[rr, "{(A,B) \mapsto (A, L_0 \oplus B)}"] \arrow[rrd, "\simeq"] & & \mathcal{S}^{E_1}(\cdot, V \subseteq \cdot \, |\, M) \dar{(P,Q) \mapsto (\tfrac{P\oplus L_0}{L_0}, \tfrac{Q}{L_0})}\\
 & & \mathcal{S}^{E_1}(\cdot, V/L_0 \subseteq \cdot \, |\, M/L_0).
\end{tikzcd}
\end{equation*}
Now $L_1$ has rank 1 so by \cite[Theorem 1.1]{Charney} $\mathcal{S}^{E_1}(\cdot, V/L_0 \subseteq \cdot \, |\, M/L_0)$ is path-connected. Thus we may find a path in it from $(\tfrac{U\oplus L_0}{L_0}, \tfrac{V}{L_0})$ to $(\tfrac{U_0\oplus L_0}{L_0}, \tfrac{V}{L_0})$. Considering this as a path in the isomorphic poset $\mathcal{S}^{E_1}(\cdot, L_1 \subseteq \cdot \, |\, U \oplus L_1)$, it starts at $(U, L_1)$ and ends at some $(U', L_1)$ for which $\tfrac{U'\oplus L_0}{L_0} = \tfrac{U_0\oplus L_0}{L_0}$. Applying the horizontal map, this gives a path from $(U, V)$ to a $(U', V)$ with $U'\oplus L_0 = U_0\oplus L_0$. Now the latter is a rank 3 direct summand of $M$ containing the pair of rank 2 direct summands $U'$ and $U_0$, and so $U' \cap U_0$ is a direct summand (by \cite[Lemma 1.2]{Charney}) of rank 1 or 2. If $U' \cap U_0$ has rank 2 then $U'=U_0$ and so we have found the required path. If $U' \cap U_0$ has rank 1 then
$$(U', V) \geq (U' \cap U_0, X) \leq (U_0, V)$$
for $X$ any choice of complement of the line $U' \cap U_0$ containing $V$. In either case we have found a path from $(U, V)$ to $(U_0, V)$, as required.

Thus \cite[Theorem 4.1]{e2cellsIII} applies. It shows that $\widetilde{\mathrm{St}}(M)$ has a filtration whose first filtration quotient is ${\mathrm{St}}(M)$, so in particular the natural map between them is surjective.
\end{proof}

\subsection{Reductive Borel--Serre spaces}

Before continuing, we introduce the reductive Borel--Serre spaces example. By the main theorem of \cite{Jansen2}, there is an $E_\infty$-algebra in topological spaces whose $\bk$-linearisation $\mathbf{RBS}$ has
$$\mathbf{RBS}(n) \simeq \bigoplus_{\substack{[M] \\ M \text{ f.g.\ projective}, \mathrm{rk}(M) = n}} \bk[|\RBS(M)|].$$
By \cite[Theorem 10.11]{Jansen2} there is an identification
$$H^{E_1}_{n,d}(\mathbf{RBS}) \cong \bigoplus_{\rk(M)=n} \tilde{H}_{d}(\Sigma T(M) /\!\!/ \GL(M);\bk),$$
where $\Sigma$ denotes the unreduced suspension, and $/\!\!/$ the pointed homotopy orbits. By the Solomon--Tits theorem, $\Sigma T(M)$ is homotopy equivalent to a wedge of $(n-1)$-spheres, and yields the Steinberg module $\mathrm{St}(M)$ as its top homology. Thus we can write this as
$$H^{E_1}_{n,d}(\mathbf{RBS}) \cong \bigoplus_{\rk(M)=n} {H}_{d-n+1}(\GL(M);\mathrm{St}(M) \otimes \bk).$$
In particular these groups vanish for $d < n-1$. This verifies axiom (I).

There is a map of $E_\infty$-algebras $\mathbf{BGL} \to \mathbf{RBS}$ (by consulting Jansen's construction), which in each grading is induced by the natural maps $B\GL(M) \to |\RBS(M)|$. As these spaces are both connected, we have
$$A_{\mathrm{Pic}(\cO)} \cong H_{*,0}(\overline{\mathbf{BGL}}) \cong H_{*,0}(\overline{\mathbf{RBS}}),$$
so $\mathbf{RBS}$ also satisfies axiom (II).

\subsection{Verifying  axioms (III) and (IV)}

We will verify these axioms for both  $\mathbf{BGL}$ and $\mathbf{RBS}$. As we have just mentioned, the $0$-truncations of $\overline{\mathbf{BGL}}$ and $\overline{\mathbf{RBS}}$ are the same, namely the discrete $E_\infty$-algebra $\gA$ presented by the quadratic algebra $A := A_{\mathrm{Pic}(\cO)}$. In particular there are $E_\infty$-algebra maps
$$\overline{\mathbf{BGL}} \lra \overline{\mathbf{RBS}} \lra \gA$$
and so induced maps
$$H^{\overline{\mathbf{BGL}}}_{n,d}(\bk) \lra H^{\overline{\mathbf{RBS}}}_{n,d}(\bk) \lra H^{\gA}_{n,d}(\bk).$$
All three groups vanish for $d<n$: the first two by assumption (I) and Lemma \ref{lem:AssIAlt} in these two examples, and the third in fact vanishes for $d \neq n$, by Corollary \ref{cor:AIsKoszul}. 

For $d=n$ these maps takes the form
\begin{equation}\label{eq:CFPMaps}
\bigoplus_{\mathclap{\rk(M)=n}} H_0(\GL(M) ; \widetilde{\mathrm{St}}(M)\otimes \bk) \to \bigoplus_{\mathclap{\rk(M)=n}} H_0(\GL(M) ; {\mathrm{St}}(M)\otimes \bk) \to \Tor^A_{n,n}(\bk, \bk).
\end{equation}

\begin{thm}\label{thm:CFP}
The right-hand map in \eqref{eq:CFPMaps} is surjective.
\end{thm}
\begin{proof}
Church, Farb, and Putman \cite[\S 5.2]{CFP} have constructed for each finitely-generated projective $\cO$-module $M$ of rank $n$ a certain  map
$$H_0(\GL(M) ;\mathrm{St}(M)) \lra \widetilde{H}_{n-2}(|X_{n-1}(\mathrm{Pic}(\cO))|;\bZ)$$
and have shown \cite[Proposition 5.5]{CFP} that it is surjective. We will explain why the right-hand map above is simply the sum, over isomorphism classes of rank $n$ projective $\cO$-modules, of these maps (tensored with $\bk$): then the claim follows.

To do so it is simplest to put $\gA$ in the framework of ``categories with filtrations'' of \cite[\S 7.1]{ClausenJansen} and rely on the functoriality of that framework. In fact it is awkward to make $\gA$ fit exactly into that setting, so we make the following observation. It is not actually necessary to have an underlying category to make many of the constructions involving ``categories with filtrations'': rather, it suffices to have a collection of objects as well as a class of ``short exact sequences'', satisfying certain properties. Among these properties one must be able to compose admissible monomorphisms, and also compose admissible epimorphisms, but for example one never needs to compose morphisms of the two kinds. Formally, one can view this as as a pair $\mathcal{C} = (\mathcal{C}^{mono}, \mathcal{C}^{epi})$ of categories with the same objects, and with certain pairs of arrows $(a \hookrightarrow b, b \twoheadrightarrow c)$ distinguished as ``short exact sequences'', again satisfying certain properties. This data suffices to define $\mathscr{M}_{RBS}(\mathcal{C})$ as in \cite[Definition 6.1]{Jansen2}.

Now define $\mathcal{A} = (\mathcal{A}^{mono}, \mathcal{A}^{epi})$ to have common set of objects the isomorphism classes of finitely-generated projective $\cO$-modules. Let
$$\mathcal{A}^{mono}([P], [Q]) := \begin{cases}
* & \rk(P) < \rk(Q) \text{ or } [P]=[Q]\\
\emptyset & \text{otherwise}
\end{cases} =: \mathcal{A}^{epi}([Q], [P]),$$
so both categories are posets. The short exact sequences are those pairs $([P] \hookrightarrow [Q], [Q] \twoheadrightarrow [R])$ such that $[Q] = [P \oplus R]$.

The category $\mathcal{P}(\cO)$ of finitely-generated projective $\cO$-modules is a category with filtrations as in \cite[Example 7.3 (3)]{ClausenJansen}, and can be considered in the form $\mathcal{P}(\cO) = (\mathcal{P}(\cO)^{mono}, \mathcal{P}(\cO)^{epi})$ as above. Sending a module to its isomorphism class then gives a morphism
$$\phi : (\mathcal{P}(\cO)^{mono}, \mathcal{P}(\cO)^{epi}) \lra (\mathcal{A}^{mono}, \mathcal{A}^{epi}).$$
This induces a map of $E_1$-algebras
\begin{equation}\label{eq:Trunc}
\mathbf{RBS} := \mathscr{M}_{RBS}(\mathcal{P}(\cO)) \lra \mathscr{M}_{RBS}(\cA).
\end{equation}
We can identify the latter as (see \cite[Definition 6.1, Notation 6.4]{Jansen2})
$$\mathrm{RBS}([P]) \simeq \begin{cases}
\text{Objects are lists } ([Q_1], \ldots, [Q_n]) \text{ of non-zero isomorphism classes}\\
\text{such that } [P] = [\bigoplus_i Q_i].\\
\text{A morphism } ([Q_1], \ldots, [Q_n]) \to ([R_1], \ldots, [R_m]) \text{ is an order preserving } \\ 
\theta: \{1,\ldots,n\} \twoheadrightarrow \{1,\ldots, m\} \text{ such that } [R_j] = [\bigoplus_{i \in \theta^{-1}(j)} Q_i].
\end{cases}$$
Observe that this is a poset, and the one-element list $([P])$ is terminal so $|\mathrm{RBS}([P])|\simeq *$. Thus $\mathscr{M}_{RBS}(\cA)$ has one contractible path component for each isomorphism class of finitely-generated projective $\cO$-module, i.e.\ it is $\gA$. The map \eqref{eq:Trunc} is thus identified with the truncation map $\mathbf{RBS} \to \gA$.

Now $\partial \mathrm{RBS}([P])$ is the full subposet of $\mathrm{RBS}([P])$ on all objects other than $[P]$. We will explain how to identify it with the poset of simplices of the the realisation of the poset $X_{rk(P)-1}(\mathrm{Pic}(\cO))$ defined in \cite[Example 5.3]{CFP}, i.e.\ with the barycentric subdivision $sd X_{rk(P)-1}(\mathrm{Pic}(\cO))$ of this poset.

Recall that the poset $X_{rk(P)-1}(\mathrm{Pic}(\cO))$ is given by $\{1,2,\ldots, rk(P)-1\} \times \mathrm{Pic}(\cO)$, with order relation $(r, [L]) < (r', [L'])$ if and only if $r<r'$. We assign to a list $([Q_1], \ldots, [Q_n])$ in the poset $\mathrm{RBS}([P])$ the chain
$$(\rk(Q_1), [\det(Q_1)]) < (\rk(Q_1 \oplus Q_2), [\det(Q_1 \oplus Q_2)]) < \cdots < (\rk(\oplus_i Q_i), [\det(\oplus_i Q_i)])$$
in $X_{rk(P)-1}(\mathrm{Pic}(\cO))$, i.e.\ an element of $sd X_{rk(P)-1}(\mathrm{Pic}(\cO))$. Conversely, we assign to a chain
$$(r_1, [L_1]) < (r_2, [L_2]) < \ldots < (r_n, [L_n])$$
in $X_{rk(P)-1}(\mathrm{Pic}(\cO))$ the list
$$([\cO^{r_1-1} \oplus L_1], [\cO^{r_2-r_1-1} \oplus (L_2 \otimes L_1^{-1})], \ldots, [\cO^{r_n-r_{n-1}-1} \oplus (L_n \otimes L_{n-1}^{-1})]).$$
This defines a bijection between the poset $\partial \mathrm{RBS}([P])$ and the barycentrically subdivided poset $sd X_{rk(P)-1}(\mathrm{Pic}(\cO))$, and one checks that it preserves the order relation.

We can form the diagram
\begin{equation*}
\begin{tikzcd}
{|\mathcal{F}(P) \setminus \emptyset| /\!\!/ \GL(P)} \rar \dar & {|\mathcal{F}(P)| /\!\!/ \GL(P)} \dar\\
{|\partial\mathrm{RBS}(P)|} \rar \dar& {|\mathrm{RBS}(P)|} \dar\\
{|\partial\mathrm{RBS}([P])|} \rar & {|\mathrm{RBS}([P])|}
\end{tikzcd}
\end{equation*}
where the top square is \cite[Corollary 5.8]{ClausenJansen} and is a homotopy pushout, and the bottom square is induced by $\phi$. As in \cite[Observation 9.15]{Jansen2}, the poset $\mathcal{F}(P) \setminus\emptyset$ is the barycentric subdivision of the Tits poset $\mathcal{T}(P)$, and with the discussion above the composition
$$|\mathcal{F}(P) \setminus \emptyset|  \lra |\mathcal{F}(P) \setminus \emptyset| /\!\!/ \GL(P) \lra |\partial\mathrm{RBS}(P)| \lra |\partial\mathrm{RBS}([P])|$$
is identified with the map induced on barycentric subdivisions by the map $\psi: \mathcal{T}(P) \to X_{rk(P)-1}(\mathrm{Pic}(\cO))$ of \cite[Section 5.2]{CFP}. Using \cite[Theorem 10.10]{Jansen2} we have
\begin{align*}
H^{E_1}_{n,n-1}(\mathbf{RBS}) &= \bigoplus_{[P], \rk(P)=n} H_{n-1}(|\mathrm{RBS}(P)|, |\partial\mathrm{RBS}(P)|) \\
&\cong \bigoplus_{[P], \rk(P)=n} H_0(\GL(P) ; \tilde{H}_{n-1}(\mathcal{T}(P)))\\
H^{E_1}_{n,n-1}(\gA) &= \bigoplus_{[P], \rk(P)=n} H_{n-1}(|\mathrm{RBS}([P])|, |\partial\mathrm{RBS}([P])|)\\
&\cong \bigoplus_{[P], \rk(P)=n} \tilde{H}_{n-1}(X_{rk(P)-1}(\mathrm{Pic}(\cO)))
\end{align*}
and the map between them is that induced by $\psi$, as required.
\end{proof}

\begin{cor}
The composition of the two maps in \eqref{eq:CFPMaps} is surjective for $n \leq 4$.
\end{cor}
\begin{proof}
Combine Theorems \ref{thm:CFP} and \ref{thm:SteinbergEpi}.
\end{proof}

\begin{cor}
Axioms (III) and (IV) hold for $\mathbf{BGL}$ and for $\mathbf{RBS}$.
\end{cor}
\begin{proof}
Apply Lemma \ref{lem:CheckHyp3and4} to the above, to see that 
all of $H_{3,2}^{\overline{\mathbf{BGL}}}(\gA)$, $H_{4,3}^{\overline{\mathbf{BGL}}}(\gA)$, $H_{3,2}^{\overline{\mathbf{RBS}}}(\gA)$, and $H_{4,3}^{\overline{\mathbf{RBS}}}(\gA)$ vanish.
\end{proof}

The axioms having been verified, Theorems \ref{thm:StabDedekind} and \ref{thm:StabDedekindRegularity} follow, as do their analogues for $|\mathrm{RBS}(M)|$.

\section{Proof of Theorem \ref{thm:RationalStab}}

We require the following algebraic lemma, a consequence of a theorem of J{\'o}zefiak and Weyman \cite{JW}.

\begin{lem}\label{lem:Matching}
Let $\bK$ be a field of characteristic zero, $V$ be a $\bK$-module, and consider the cdga
$$C_{\bullet, *} := \left(\mathrm{Sym}^*(V[1,0]) \otimes \Lambda^*(\mathrm{Sym}^2(V)[2,1]), \partial\right)$$
with additional $\bN$-grading denoted $\bullet$, where $\partial$ is determined by the identity map $C_{2,1} = \mathrm{Sym}^2(V) \to C_{2,0} = \mathrm{Sym}^2(V)$ and the graded Leibniz rule. Then
$$H_{n,d}(C_{\bullet, *})=0 \text{ for } d < \tfrac{n-\sqrt{n}}{2}.$$

If $V$ is finite-dimensional then $H_{n,d}(C_{\bullet, *})=0$ for $d < \tfrac{n-\dim(V)}{2}$. If $n$ is a square, $d = \tfrac{n-\sqrt{n}}{2}$, and $\dim(V) \geq n-2d$, then $H_{n,d}(C_{\bullet, *}) \neq 0$.
\end{lem}
\begin{proof}
J{\'o}zefiak and Weyman \cite{JW} (see \cite[Theorem 2.5]{KRW} for the following formulation) identify the homology of $C_{\bullet, *}$ in terms of Schur functors of $V$, as
$$H_{n,d}(C_{\bullet, *}) \cong \bigoplus_{\substack{\lambda \\ |\lambda|=n, \lambda = \lambda^t, d(\lambda) = n-2d}} S_{\lambda}(V)$$
where $d(\lambda)$ denotes the length of the diagonal when $\lambda$ is viewed as a Young diagram. In order for there to be any summands at all, we must have that $n \geq (n-2d)^2$ (as if $\lambda$ has diagonal of length $n-2d$ then it contains the $(n-2d) \times (n-2d)$ square), i.e.\ that $d \geq \tfrac{n-\sqrt{n}}{2}$. Thus it vanishes if $d < \tfrac{n-\sqrt{n}}{2}$. 

If $V$ is finite-dimensional then $S_\lambda(V)=0$ if and only if the number of rows of $\lambda$ is greater than $\dim(V)$. If $d(\lambda)=n-2d$ then it must have at least $n-2d$ rows, so $S_\lambda(V)=0$ when $n-2d > \dim(V)$, i.e.\ when $d < \tfrac{n-\dim(V)}{2}$.

If $n$ is a square and $d = \tfrac{n-\sqrt{n}}{2}$, so $n = (n-2d)^2$, then taking $\lambda$ to be the $(n-2d) \times (n-2d)$ square we see that $S_\lambda(V) \neq 0$ when $\dim(V) \geq n-2d$, which contributes to $H_{n,d}(C_{\bullet, *})$.
\end{proof}

\begin{rem}
Over a field $\bK$ of positive characteristic, choosing a basis $\{b_\alpha\}_{\alpha \in I}$ of $V$, a basis for $C_{\bullet, *}$ may be given by symbols
$$b_{\alpha_1} \cdots b_{\alpha_i} \otimes (b_{\alpha_{i+1}} \cdot b_{\alpha_{i+2}}) \wedge \cdots \wedge (b_{\alpha_{i+2k+1}} \cdot b_{\alpha_{i+2k+2}}).$$
The differential is given by an alternating sum of moving each wedge-summand across the tensor: in particular the multiset of subscripts $\alpha_i$ is unchanged by the differential. Thus $C_{n,*}$ may be described as a direct sum of chain complexes, one for each $n$-element multiset in $I$. If $n \leq |I|$ then there is a summand given by an $n$-element multiset of $I$ with no repeated elements, i.e.\ an injection $\{1,2,\ldots, n\} \to I$. The corresponding direct summand of $C_{n,*}$ may be identified with the 1-fold suspension of the augmented simplicial chain complex of the \emph{matching complex} $M(\{1,2,\ldots, n\})$ on the set $\{1,2,\ldots, n\}$, see e.g.\ \cite{KRW}. So as long as $\dim(V) \geq n$ then
$$\widetilde{H}_{d-1}(M(\{1,2,\ldots, n\}) ; \bK) \subset H_{n,d}(C_{\bullet, *}).$$
Now Bouc \cite[Proposition 7]{Bouc} has shown that
$$\widetilde{H}_{k}(M(\{1,2,\ldots, 3k+4\}) ; \bZ) \cong \bZ/3 \text{ for all } k \geq 1,$$
so $H_{3d+1,d}(C_{\bullet, *};\bK) \neq 0$ if $\bK$ has characteristic 3. Thus Lemma \ref{lem:Matching} cannot hold in positive characteristic (or at least not in characteristic 3).
\end{rem}

\begin{proof}[Proof of Theorem \ref{thm:RationalStab}]
Recall that $\bk$ is a field of characteristic zero. We construct a CW-approximation to $\gR$ by
$$\gC := \gE_\infty(S^{1,0} \otimes \bk\{P\}) \cup^{E_\infty}_{\rho \cdot \rho' - \sigma \cdot (\rho \star \rho')} D^{2,1} \lra \gR,$$
where we attach a $(2,1)$-cell for each pair $(\rho, \rho') \in P^2$: the map to $\gR$ extends over these cells as $[\rho] \cdot [\rho'] = [\sigma] \cdot [\rho \star \rho'] \in H_{2,1}(\gR)$ by assumption (II). The map $H_{*,0}(\overline{\gC}) \to H_{*,0}(\overline{\gR}) = A_P$ is an isomorphism, by construction. As both $\gR$ and $\gC$ satisfy assumption (I), it follows from \cite[Theorem 14.4]{e2cellsI} and the long exact sequence on $E_\infty$-homology for this map that $H_{n,d}^{E_\infty}(\gR, \gC)=0$ for $d < \tfrac{1}{2}n$. We may therefore attach $(n,d)$-cells with $d \geq \tfrac{1}{2}n$ to $\gC$ to obtain a CW-approximation $\gC \to \gC' \overset{\sim}\to \gR$, equipped with a relative skeletal filtration. We will treat stability with respect to $\sigma \in P$, but the argument can be easily adapted to any other $\rho\in P$: the only change is notation.

As, for example, in the proof of \cite[Theorem 18.3]{e2cellsI}, this leads to a spectral sequence
$$E^1_{n,p,q} = H_{n,p+q,q}(\overline{\gC}/\sigma[0] \otimes E_\infty^+(\bigoplus_{\alpha \in I} S^{n_\alpha, d_\alpha, d_\alpha})) \Longrightarrow H_{n,p+q}(\overline{\gR}/\sigma),$$
with $d_\alpha \geq \tfrac{1}{2} n_\alpha$. As the homology of $E_\infty^+(\bigoplus_{\alpha \in I} S^{n_\alpha, d_\alpha, d_\alpha})$ is supported in slopes $\geq \tfrac{1}{2}$, whatever vanishing is enjoyed by $\overline{\gC}/\sigma$ will also hold for $\overline{\gR}/\sigma$, up to slope $\tfrac{1}{2}$.

We now analyse $\overline{\gC}/\sigma$, using Lemma \ref{lem:Matching}. Letting $W:= \bk\{P\}$, the homology of $\overline{\gC}$ is calculated by the bigraded cdga
$$D_{\bullet, *} := (\mathrm{Sym}^*(W[1,0]) \otimes \Lambda^*(W \otimes W [2,1]), \partial)$$
with differential determined by 
$$\rho \otimes \rho' \mapsto \rho \cdot \rho' - \sigma \cdot (\rho \star \rho') : D_{2,1} = W \otimes W \lra D_{2,0} =  \mathrm{Sym}^2(W)$$
and the Leibniz rule. As in the first few lines of the proof of \cite[Theorem 9.5]{e2cellsIV}, as $\bk$ has characteristic zero there is an equivalence $\overline{\gC}/\sigma \simeq \overline{\gC \cup^{E_\infty}_\sigma D^{1,1}}$. Letting $V := \bk\{P \setminus \{\sigma\}\}$, the homology of $\overline{\gC \cup^{E_\infty}_\sigma D^{1,1}}$ may be calculated by the bigraded cdga
$$D'_{\bullet, *} := (\mathrm{Sym}^*(V[1,0]) \otimes \Lambda^*(W \otimes W [2,1]), \partial)$$
with differential determined by
$$\rho \otimes \rho' \mapsto \begin{cases}
\rho \cdot \rho' & \rho \neq \sigma \text{ and } \rho' \neq \sigma\\
0 & \rho=\sigma \text{ or } \rho' = \sigma
\end{cases}$$
and the Leibniz rule. Now in terms of the cdga $C_{\bullet, *}$ introduced in Lemma \ref{lem:Matching} we have $D'_{\bullet, *} = C_{\bullet, *} \otimes \{\substack{\text{a free graded commutative algebra with} \\ \text{generators of slope $\tfrac{1}{2}$ and trivial differential}}\}$, and so
$$H_{*,*}(\overline{\gC}/\sigma) \cong H_{*,*}(C_{\bullet, *})\otimes \{\substack{\text{a free graded commutative algebra} \\ \text{with generators of slope $\tfrac{1}{2}$}}\}.$$
It then follows from Lemma \ref{lem:Matching} that
\begin{enumerate}[(i)]
\item $H_{n,d}(\overline{\gC}/\sigma)=0$ for $d < \tfrac{n-\sqrt{n}}{2}$, and

\item if $P$ is finite then $H_{n,d}(\overline{\gC}/\sigma)=0$ for $d < \tfrac{n-(|P|-1)}{2}$.
\end{enumerate}
With the discussion above, this finishes the argument.
\end{proof}

\begin{rem}\label{rem:sqrtEstimateSharp}
The $E_\infty$-algebra $\gC$ constructed in this argument satisfies assumptions (I) and (II) of Theorem \ref{thm:StabGeneric}, but the argument shows that $H_{*,*}(\overline{\gC}/\sigma)$ contains the homology of the cdga of Lemma \ref{lem:Matching} with $V = \bk\{P \setminus \{\sigma\}\}$. By the last part of that lemma, the homology of this cdga does not vanish in bidegrees $(n,d)$ when $n$ is a square, $d = \tfrac{n-\sqrt{n}}{2}$, and $|P|-1 \geq n-2d$. In particular if $P$ is infinite then $H_{*,*}(\overline{\gC}/\sigma)$ does not admit a vanishing line of slope $\tfrac{1}{2}$.
\end{rem}

\section{An example}\label{sec:Example}
Consider the ring of integers $\cO = \bZ[\sqrt{-5}]$ of the number field $\bQ(\sqrt{-5})$, and write $\omega := \sqrt{-5}$. This field has class number 2, and so $\mathrm{Pic}(\cO) = \{\sigma, \lambda\}$ is the group of order 2: its nontrivial element $\lambda$ may be represented by the (fractional) ideal $\mathfrak{l} = (2, 1+\omega)$.

\vspace{1ex}

\noindent\textbf{Rank 1.} Both $\GL(\cO)$ and $\GL(\mathfrak{l})$ are given by $\cO^\times = \{\pm 1\}$.

\vspace{1ex}

\noindent\textbf{Rank 2.} A presentation for the group $\mathrm{SL}_2(\cO)$ has been given by Swan \cite[Theorem 11.1]{Swan}: it is generated by the matrices
$$
J= \begin{pmatrix}
-1 & 0\\
0 & -1
\end{pmatrix}, T=\begin{pmatrix}
1 & 1\\
0 & 1
\end{pmatrix},U=\begin{pmatrix}
1 & \omega\\
0 & 1
\end{pmatrix}$$
$$A = \begin{pmatrix}
0 & -1\\
1 & 0
\end{pmatrix}, B=\begin{pmatrix}
-\omega & 2\\
2 & \omega
\end{pmatrix}, C=\begin{pmatrix}
-\omega-4 & -2\omega\\
2\omega & \omega-4
\end{pmatrix}$$
subject to the relations
$$J^2 = 1, J \text{ central}, TU=UT, A^2=J, B^2=J, (TA)^3=J, (AB)^2=J$$
$$(AUBU^{-1})^2=J, ACA=JTCT^{-1}, UBU^{-1}CB = JTCT^{-1}.$$
This abelianises to $\bZ/2 \oplus \bZ/6 \oplus \bZ \oplus \bZ$, generated for example by $B, T , U, C$ (\cite[Corollary 11.2]{Swan}). Furthermore, $J \sim 0$ in this group. The matrix $E := \left( \begin{smallmatrix}
-1 & 0\\
0 & 1
\end{smallmatrix} \right)$ acts on $\mathrm{SL}_2(\cO)$ by conjugation, as $B \mapsto JUBU^{-1} \sim B$, $T \mapsto T^{-1} \sim -T$, $U \mapsto U^{-1} \sim -U$, and $C \mapsto TC^{-1}T^{-1} \sim -C$. This shows that
$$\GL(\cO \oplus \cO)^{ab} = (\bZ/2)^5$$
generated by $E, U, T, B, C$ (\cite[Corollary 11.3]{Swan}). 

It remains to describe the abelianisation of $\GL(\cO \oplus \mathfrak{l})$. Using that $\mathfrak{l}^{-1} = \tfrac{1}{2}\mathfrak{l}$, we may describe $\GL(\cO \oplus \mathfrak{l})$ as the subgroup:
$$\left\{\begin{pmatrix}
a & \tfrac{1}{2}b\\
c & d
\end{pmatrix} \in \GL_2(\bQ(\omega)) \, | \, a, d \in \cO, b,c \in \mathfrak{l}\right\} \leq \GL_2(\bQ(\omega)).$$ 
In \cite[Section 2.2]{FGT} we find a presentation for $\mathrm{PSL}(\cO \oplus \mathfrak{l})$ which is easily lifted to the following presentation of $\mathrm{SL}(\cO \oplus \mathfrak{l})$: it is generated by the matrices
$$J= \begin{pmatrix}
-1 & 0\\
0 & -1
\end{pmatrix}, 
A = \begin{pmatrix}
1 & 1\\
0 & 1
\end{pmatrix},
V = \begin{pmatrix}
1 & \tfrac{1+\omega}{2}\\
0 & 1
\end{pmatrix},
C = \begin{pmatrix}
1 & 0\\
2 & 1
\end{pmatrix},
D = \begin{pmatrix}
1 & 0\\
1-\omega & 1
\end{pmatrix}$$
subject to the relations
$$J^2 = 1, J \text{ central}, AV=VA, CD=DC, (AC^{-1})^2 = J, (DV^{-1})^3=1, (CD^{-1}VA^{-1})^3=1.$$
This presentation abelianises to $\bZ/3 \oplus \bZ \oplus \bZ$, generated by $DV^{-1}$, $C$, and $D$. The matrix $E$ above acts by conjugation, as $DV^{-1} \mapsto D^{-1}V \sim -(DV^{-1})$, $C \mapsto C^{-1} \sim -C$, and $D \mapsto D^{-1} \sim -D$. This gives
$$\GL(\cO \oplus \mathfrak{l})^{ab} = (\bZ/2)^3$$
generated by $E$, $C$, and $D$. Furthermore, $J \sim 0$ in this group.

\vspace{1ex}

\noindent\textbf{Rank 3.} By Bass--Milnor--Serre \cite[Corollary 4.3 a)]{BMS} the group $\mathrm{SL}_3(\cO)$ is perfect, so the map $\det : \GL_3(\cO) \to \cO^\times = \{\pm 1\}$ is the abelianisation. As $[\mathfrak{l}] \in \mathrm{Pic}(\cO)$ has order 2, tensoring with it gives an isomorphism
$$\mathfrak{l} \otimes - : \GL(\cO \oplus \cO \oplus \mathfrak{l}) \overset{\sim}\lra \GL(\mathfrak{l} \oplus \mathfrak{l} \oplus \mathfrak{l}^{\otimes 2}) \cong \GL(\cO \oplus \cO \oplus \cO)=\GL_3(\cO),$$
and so $\det : \GL(\cO \oplus \cO \oplus \mathfrak{l}) \to \cO^\times = \{\pm 1\}$ is the abelianisation too.

\bibliographystyle{amsalpha}
\bibliography{MainBib}

\end{document}